\documentclass[final,reqno]{amsart}

\usepackage{graphicx}
\usepackage{amsmath,amsthm,amsfonts,amscd,amssymb}
\usepackage[color,notref,notcite]{showkeys}
\usepackage{url}
\usepackage{subeqnarray}
\usepackage[color,notref,notcite]{showkeys} 
\definecolor{labelkey}{rgb}{0.6,0,1}
\usepackage{enumitem}
\usepackage{algorithm}
\usepackage{algpseudocode}
\usepackage{citesort}
\usepackage{esint}

\theoremstyle{plain}
\newtheorem{theorem}{Theorem}[section]

\theoremstyle{definition}
\newtheorem{definition}[theorem]{Definition}

\def\bhyp#1{\begin{equation}\label{#1}\begin{array}{c}}
\def\ehyp{\end{array}\end{equation}}

\newcounter{cst}

\theoremstyle{remark}

\numberwithin{equation}{section}
\numberwithin{figure}{section}

\newcommand{\RR}{{\mathbb R}}
\newcommand{\NN}{{\mathbb N}}

\def\O{{\Omega}}
\def\dsp{\displaystyle}

\def\disc{{\mathcal D}}

\def\dr{\partial}

\newcommand{\cK}{{\mathcal K}}

\DeclareMathOperator*{\argminB}{argmin}

\newcommand{\x}{\pmb{x}}

\newif\ifcorr\corrtrue
\usepackage[normalem]{ulem}
\definecolor{violet}{rgb}{0.580,0.,0.827}

\newcommand{\ud}{\, \mathrm{d}} 

\newcommand{\cC}{{\mathbb C}}
\newcommand{\cS}{{\mathbb S}}
\newcommand{\cW}{{\mathbb W}}

\title[Analysis of Schemes for Fourth Order VI]{A new Discrete Analysis Of Fourth Order Elliptic Variational Inequalities}

\author{Yahya Alnashri\\}
\address[Yahya Alnashri]{Department of Mathematics, Al-Qunfudah University College, Umm Al-Qura University, Saudi Arabia}
\email{yanashri@uqu.edu.sa}
\date{\today}

\begin{document}

\begin{abstract}
This paper applies the gradient discretisation method (GDM) for fourth order elliptic variational inequalities. The GDM provides a new formulation of error estimates and a complete discrete analysis of different approximating schemes. We show that the convergence is unconditional. Classical assumptions on data are only sufficient to establish the convergence results. These results are applicable to all schemes that fall in the framework of GDM.
\end{abstract}
\maketitle

\section{Introduction}
Fourth order variational inequalities are used to model numerous problems arising in mechanics and physics \cite{N-21-Louis,Stampacchia-2000,JF-1987}. Mathematical results corresponding to the well-posedness, stability, and regularity of the solutions to obstacle problems can be found in \cite{ES-1,Stampacchia-2000,R-1,R-2,R-3,R-4,R-5,Free-1988}.

Numerically, fourth order variational inequalities have been approximated by several schemes. \cite{F-1974} establishes a generic convergence rate of conforming methods applied to variational inequalities. With smooth data, \cite{Brezz-1977} drives the best error estimate for quadratic and linear finite element methods. Error estimates for discontinuous Galerkin methods have been provided in \cite{DG-2019}. Without studying the convergence rates, \cite{Francesco-1980} solves variational inequalities by mixed finite element methods. \cite{Wang-1990,Wang-1992} generates a convergence order for the non conforming finite element method.

Although different studies apply finite volume methods (as in \cite{Droniou-2010,Antonietti-2014,Antonietti-13,FV-2001,crouzeix-raviart}) to second order variational inequalities, it seems that such these schemes have not yet been designed for the fourth order variational inequalities. 

The purpose of this work is to extend the gradient discretisation methods \cite{30} to elliptic variational inequalities with a fourth order operator to obtain general error estimates and convergence analysis that hold for different conforming and non conforming numerical methods.

The outlines of this paper are as follows: Section \ref{sec-model} states the formulation of continuous problems and its full discrete scheme. Section \ref{sec-results} introduces and proves the main theoretical results, error estimates, and convergences. 

\section{Continuous and discrete setting}\label{sec-model}
Let $\O\subset \RR^d$ ($d>1$) be a bounded connected domain with boundary $\dr\O$, $f\in L^2(\O)$, the barrier $\psi \in C^2(\O)\cap C(\overline \O)$ and $\psi \geq 0$ on $\dr\O$. In this paper, we study here the following fourth order variational inequality: Seek $\bar c \in \mathcal K$ satisfying

\begin{equation}\label{weak-obs}
\dsp\int_\O \Delta \bar c(\x)\Delta(\bar c-\varphi)(\x) \ud\x \leq \dsp\int_\O f(\x)(\bar c(\x)-\varphi(\x))\ud \x, \quad \forall \varphi\in \cK,
\end{equation}
where the non empty closed convex set $\cK$ is defined by
\[
\cK:=\{\varphi\in H_0^2(\O)\;:\; \varphi \leq \psi \mbox{ in } \O\}.
\]
Note that the standard theory founded in \cite{Stampacchia-2000} shows that the above problem is well-posed. We see that the model can be formulated by the following equivalent energy minimization equation:
\[
\mbox{Find $\bar c\in \cK$ such that }, \quad \argminB_{\varphi\in\cK}\frac{1}{2}\dsp\int_\O \Delta \bar c(\x)\Delta \varphi(\x) \ud\x - \dsp\int_\O f(\x)\varphi(\x)\ud \x .
\] 

Here, we define the discrete elements (called gradient discretisation) to construct the approximation scheme for our problem.
 
\begin{definition}\label{def-GD}
Let $\O$ be an open domain of $\RR^d$ ($d> 1$). A gradient discretisation $\disc$ for fourth order obstacle problems is give by a family $\disc=(X_{\disc,0},\Pi_\disc,\nabla_\disc,\Delta_\disc)$, where:
\begin{itemize}
\item The discrete set $X_{\disc,0}$ is a finite-dimensional vector space on $\RR$, dealing with the unknowns of the method.
\item The linear operator $\Pi_\disc : X_{\disc,0} \to L^2(\O)$ is the reconstruction of the approximate function.
\item The linear operator $\nabla_\disc : X_{\disc,0} \to L^2(\O)^d$ is the reconstruction of the gradient of the function.
\item $\Delta_\disc: X_{\disc,0} \to L^2(\O)$ is a linear mapping to construct a discrete of the bi-harmonic $\Delta$ form, and must be defined so that $\| \Delta_\disc\cdot \|_{L^2(\O)}$ is a norm on $X_{\disc,0}$.
\end{itemize}
\end{definition}

\begin{definition}
Let $\disc$ be a gradient discretisation. The discrete version of the priblem\eqref{weak-obs} is given by
\begin{equation}\label{disc-obs}
\left.
\begin{array}{ll}
&\mbox{ Find $c\in \cK_\disc:=\{\varphi \in X_{\disc,0}\; :\; \Pi_\disc \varphi \leq \psi\}$ such that, for all $\varphi\in \cK_\disc$},\\
&\dsp\int_\O \Delta_\disc c(\x)\Delta_\disc(c-\varphi)(\x) \ud\x \leq \dsp\int_\O f(\x)\Pi_\disc(c(\x)-\varphi(\x))\ud \x.
\end{array}
\right.
\end{equation}
\end{definition}

The accuracy of this approximation can be measured by the following three indicators; the first one is the constant $\cC_\disc$, which measures the coercivity and is defined by
\begin{equation}\label{corc-obs}
\cC_\disc = \dsp\max_{\omega \in X_{\disc,0}\setminus\{0\}}\left( \frac{\| \Pi_\disc \omega \|_{L^2(\O)}}{\| \Delta_\disc \omega\|_{L^2(\O)}},\; \frac{\| \nabla_\disc \omega \|_{L^2(\O)^d}}{\| \Delta_\disc \omega \|_{L^2(\O)}} \right).
\end{equation}
It implies the discrete Poincar\'e inequalities, for any $w\in X_{\disc,0}$,
\begin{equation}\label{eq-ponc}
\| \Pi_\disc \omega \|_{L^2(\O)} \leq \cC_\disc \| \Delta_\disc \omega\|_{L^2(\O)} \mbox{ and }
\| \nabla_\disc \omega \|_{L^2(\O)^d} \leq \cC_\disc \| \Delta_\disc \omega\|_{L^2(\O)}.
\end{equation}
The second quantity is the function $\cS_\disc:\cK \to \cK_\disc$, which measures the interpolation error and it is given by: For all $v \in \cK$, 
\begin{equation}\label{const-obs}
\begin{aligned}
\cS_\disc(v)=\dsp\min_{\omega\in \cK_\disc}
\Big( \| \Pi_\disc \omega &- v \|_{L^2(\O)} + \|\nabla_\disc \omega - \nabla v \|_{L^2(\O)^d}\\
&\quad+ \| \Delta_\disc \omega - \Delta v \|_{L^2(\O)}  \Big).
\end{aligned}
\end{equation}
The last one is the function $\cW_\disc:H_\Delta(\O) \to \RR$, which refers to the limit--conformity and it is defined by: For $v \in H_\Delta(\O)$,
\begin{equation}\label{conf-obs}
\cW_\disc(v)=\dsp\max_{\omega\in X_{\disc,0}\setminus\{0\}}
\frac{1}{\| \Delta_\disc \omega \|_{L^2(\O)}}\dsp\int_\O \big( \Delta v \Pi_\disc \omega - \dsp\int_\O v \Delta_\disc \omega \big) \ud \x,
\end{equation}
where $H_\Delta(\O):=\left\{ v \in L^2(\O)\;:\;\Delta v \in L^2(\O) \right\}$.

\begin{definition}
Let $(\disc_m)_{m\in\NN}$ be a sequence of a gradient discretisation in the sense of definition \ref{def-GD}. We say that
\begin{itemize}
\item $(\disc_m)_{m\in\NN}$ is coercive if there exists $\cC_P\in \RR^+$ such that $\cC_{\disc_m} \leq \cC_P$ for all $m \in \NN$.

\item $(\disc_m)_{m\in\NN}$ is consistent if
\[
\mbox{for all } v \in \cK, \quad \lim_{m\to\infty}\cS_\disc(v)=0.
\]

\item $(\disc_m)_{m\in\NN}$ is limit--conforming if
\[
\mbox{for all } v \in H_\Delta(\O), \quad \lim_{m\to\infty} \cW_\disc(v)=0.
\]

\item $(\disc_m)_{m\in\NN}$ is compact if for any sequence $(\omega_m)_{m\in\NN}$ in $X_{{\disc_m},0}$, such that $\| \Delta_{\disc_m} \omega_m \|_{L^2(\O)}$ is a bounded, then there exists $\varphi \in H_0^1(\O)$, such that the sequence $(\Pi_{\disc_m}\omega_m)_{m\in\NN}$ converges strongly to $\varphi$ in $L^2(\O)$, as $m \to \infty$. 
\end{itemize}
\end{definition}

\section{Main results}\label{sec-results}
Let us now introduce the main results; the general error estimates and the convergence results. We begin with defining the continuous interpolant $I_\disc: \cK \to \cK_\disc$ by
\begin{equation}\label{interp-eq}
\begin{aligned}
I_\disc h=\argminB_{\omega\in \cK_\disc}\big( \| \Pi_\disc \omega - h \|_{L^2(\O)} &+ \| \nabla_\disc \omega - \nabla h \|_{L^2(\O)^d}\\ 
&+ \| \Delta_\disc w - \Delta h \|_{L^2(\O)} \big).
\end{aligned}
\end{equation}
Thus, from the definition of $\cS_\disc$, we have, for any $v \in \cK$,
\begin{equation}
\| \Pi_\disc I_\disc v - v \|_{L^2(\O)} + \| \nabla_\disc I_\disc v - \nabla v \|_{L^2(\O)^d} + \| \Delta_\disc I_\disc v - \Delta v \|_{L^2(\O)}
\leq \cS_\disc(v).
\label{interpl}
\end{equation}

\begin{theorem}[Convergence rates]\label{them-error}
Assume that $\O\subset \RR^d\; (d>1)$ is a bounded domain with the boundary $\dr\O$, $f \in L^2(\O)$, $\psi \in C^2(\O)\cap C(\overline \O)$, $\psi \geq 0$ on $\dr\O$, and $\bar c$ be the solution to \eqref{weak-obs}. Let $\disc$ be a gradient discretisation such that $\cK_\disc \neq \Phi$, then the discrete scheme \eqref{disc-obs} has a unique solution $c \in \cK_\disc$. Moreover, if it is assumed that $\Delta^2 \bar c \in L^2(\O)$, then: 

\begin{equation}
\begin{aligned}
\| \Pi_\disc c &- \bar c \|_{L^2(\O)}\\
&\leq
\cC_\disc\frac{\sqrt 2}{2}\cW_\disc(\Delta \bar c)+(\cC_\disc\frac{\sqrt 2+2}{2}+1)\cS_\disc(\bar c)+\cC_\disc R_\disc(\bar c)^{\frac{1}{2}},
\end{aligned}\label{est-1}
\end{equation}
\begin{equation}
\begin{aligned}
\| \nabla_\disc c &- \nabla\bar c \|_{L^2(\O)^d}\\
&\leq
\cC_\disc\frac{\sqrt 2}{2}\cW_\disc(\Delta \bar c)+(\cC_\disc\frac{\sqrt 2+2}{2}+1)\cS_\disc(\bar c)+\cC_\disc R_\disc(\bar c)^{\frac{1}{2}},
\end{aligned}\label{est-2}
\end{equation}
\begin{equation}\label{est-3}
\| \Delta_\disc c - \Delta \bar c \|_{L^2(\O)}
\leq 
\frac{\sqrt 2}{2}W_\disc(\Delta \bar c)+\frac{\sqrt 2+2}{2}\cS_\disc(\bar c)+R_\disc(\bar c)^{\frac{1}{2}},
\end{equation}
where $R_\disc(\bar c):=\dsp\int_\O (\Delta^2 \bar c +f)(\psi - \Pi_\disc I_\disc \bar c)\ud \x$. 
\end{theorem}

\begin{proof}
The existence and uniqueness of the discrete solution follow from Stampacchia's theorem \cite{Stampacchia-2000}, thanks to the assumption that $\cK_\disc$ is a non empty set. Under the regularity $\Delta^2\bar c \in L^2(\O)$, we can apply \eqref{conf-obs} to $v:=\Delta\bar c \in H_\Delta(\O)$ with taking $w:=I_\disc \bar c - c \in X_{\disc,0}$ to deduce
\begin{equation*}
\begin{aligned}
\int_\O \Delta\bar c(\x) \Delta_\disc(I_\disc \bar c - c)(\x) \ud \x
&+\dsp\int_\O \Delta^2\bar c(\x) \Pi_\disc (I_\disc \bar c - c)(\x) \ud \x\\
&\leq
\|\Delta_\disc(I_\disc \bar c - c)\|_{L^2(\O)}
\cW_\disc(\Delta\bar c),
\end{aligned}
\end{equation*}
where $I_\disc$ is the interpolant defined by \eqref{interp-eq}. Since $c$ is the solution to \eqref{disc-obs}, the above inequality implies 
\begin{equation}\label{new-1}
\begin{aligned}
&\int_\O\Delta_\disc(I_\disc \bar c - c)(\x)(\Delta\bar c(\x)-\Delta_\disc c(\x)) \ud \x
+\dsp\int_\O \Delta^2\bar c(\x) \Pi_\disc (I_\disc \bar c - c)(\x) \ud \x\\
&\leq\|\Delta_\disc(I_\disc \bar c - c)\|_{L^2(\O)}
\cW_\disc(\Delta\bar c)
-\dsp\int_\O f(\x) \Pi_\disc(I_\disc \bar c - c)(\x) \ud \x.
\end{aligned}
\end{equation}
Thanks to the regularity assumption $\Delta^2\bar c \in L^2(\O)$, taking $\varphi:=\bar c-v$ (with a non negative $v \in C_0^\infty(\O)$) as a generic function in \eqref{weak-obs} shows that $(f+\Delta^2\bar c) \geq 0$ for a.e. in $\O$. Thus, we have
\begin{equation*}
\begin{aligned}
&\dsp\int_\O \Delta^2\bar c(\x) \Pi_\disc (I_\disc \bar c - c)(\x) \ud \x\\
&=\dsp\int_\O (\Delta^2\bar c(\x) +f(\x)) \Pi_\disc (I_\disc \bar c - c)(\x) \ud \x
-\dsp\int_\O f(\x) \Pi_\disc (I_\disc \bar c - c)(\x) \ud \x\\
&=\dsp\int_\O (\Delta^2\bar c(\x) +f(\x)) (\Pi_\disc I_\disc \bar c(\x) - \psi(\x)) \ud \x
+\dsp\int_\O (\Delta^2\bar c(\x) +f(\x)) (\psi(\x) - \Pi_\disc I_\disc \bar c(\x))\\
&\quad-\dsp\int_\O f(\x) \Pi_\disc (I_\disc \bar c - c)(\x) \ud \x\\
&\geq \dsp\int_\O (\Delta^2\bar c(\x) +f(\x)) (\Pi_\disc I_\disc \bar c(\x) - \psi(\x)) \ud \x
-\dsp\int_\O f(\x) \Pi_\disc (I_\disc \bar c(\x) - c(\x)) \ud \x,
\end{aligned}
\end{equation*}
since the quantity $\dsp\int_\O (\Delta^2\bar c(\x) +f(\x)) (\psi(\x) - \Pi_\disc I_\disc \bar c(\x)) \ud \x \geq 0$. Substituting the above relation in \eqref{new-1} yields
\begin{equation*}
\int_\O\Delta_\disc(I_\disc \bar c(\x)-c(\x))(\Delta\bar c(\x)-\Delta_\disc c(\x)) \ud \x
\leq\|\Delta_\disc(I_\disc \bar c - c)\|_{L^2(\O)}
\cW_\disc(\Delta\bar c)+R_\disc(\bar c).
\end{equation*}
Introduce the term $\Delta_\disc I_\disc \bar c$ and apply the Cauchy--Schwarz's inequality to obtain
\begin{equation}
\begin{aligned}
\|\Delta_\disc (I_\disc \bar c &-c) \|^{2}\\
&\leq
\int_\O\Delta_\disc(I_\disc \bar c(\x) - c(\x))(\Delta_\disc I_\disc \bar c(\x)-\Delta\bar c(\x)) \ud \x\\ 
&\quad+\|\Delta_\disc(I_\disc \bar c - c)\|_{L^2(\O)}
\cW_\disc(\Delta\bar c)+R_\disc(\bar c)^+\\
&\leq \| \Delta_\disc(I_\disc \bar c - c) \|_{L^2(\O)}
\big( \| \Delta_\disc I_\disc \bar c-\Delta\bar c\|_{L^2(\O)} 
+\cW_\disc(\Delta\bar c)\big)\\
&\quad+R_\disc(\bar c)^+,
\end{aligned}
\label{eq-100}
\end{equation}
which leads to, thanks to \eqref{interpl}
\begin{equation*}
\begin{aligned}
\|\Delta_\disc (I_\disc \bar c-c) \|_{L^2(\O)}^{2}
\leq
\|\Delta_\disc (I_\disc \bar c - c) \|_{L^2(\O)}
\big(\cW_\disc(\Delta\bar c)
+\cS_\disc(\bar c)\big)
+R_\disc(\bar c)^+.
\end{aligned}
\end{equation*}
Applying Young's inequality to this relation gives
\begin{equation}
\|\Delta_\disc I_\disc \bar c - \Delta_\disc c\|_{L^2(\O)} 
\leq
\big(
\frac{1}{2}\left(\cW_\disc(\Delta\bar c)+\cS_\disc(\bar c)\right)^{2}+R_\disc(\bar c)^+ \big)^{\frac{1}{2}}.
\label{proof1obs}
\end{equation}
This inequality with the use of triangle inequality, the fact that $\forall a,b \in \RR^+,\;(a+b)^{\frac{1}{2}}\leq a^{\frac{1}{2}} + b^{\frac{1}{2}}$ and \eqref{interpl} yield Estimate \eqref{est-1}. Using the definition of $\cC_\disc$, and \eqref{proof1obs}, and \eqref{interpl}, one has  
\begin{equation*}
\|\Pi_\disc I_\disc \bar c-c\|_{L^2(\O)} \leq
\cC_\disc \big(
\frac{1}{2}\left(\cW_\disc(\Delta\bar c)+\cS_\disc(\bar c)\right)^{2}+R_\disc(\bar c)^+ \big)^{\frac{1}{2}}
+\cS_\disc(\bar c).
\end{equation*}
Applying $(a+b)^{\frac{1}{2}}\leq a^{\frac{1}{2}} + b^{\frac{1}{2}}$ again establishes Estimate \eqref{est-3}. Estimate \eqref{est-2} follows in a similar way.
\end{proof}

From Theorem \ref{them-error}, we can obtain an optimal convergence rate for the approximation of fourth order variational inequalities in terms of the mesh size. \cite[Remark 2.24]{30} provides the link between the mesh size and the parameters $\cC_\disc$, $\cS_\disc$, and $\cW_\disc$ that appear in Estimates \eqref{est-1}--\eqref{est-3}. Under the assumptions stated in the previous theorem, it is also clear to establish an estimate on the term $R_\disc(\bar c)$ sinse it can be rewritten as
\[
\begin{aligned}
R_\disc(\bar c)&=\dsp\int_\O (\Delta^2 \bar c +f)(\psi - \Pi_\disc I_\disc \bar c) \ud \x\\
&=\dsp\int_\O (\Delta^2 \bar c +f)(\psi - \bar c) \ud \x
+\dsp\int_\O (\Delta^2 \bar c +f)(\bar c - \Pi_\disc I_\disc \bar c) \ud \x\\
&\leq \|\Delta^2 \bar c +f\|_{L^2(\O)} 
S_\disc(\bar c).
\end{aligned}
\]

Note that the regularity assumption that $\Delta^2\bar c \in L^2(\O)$ in the previous theorem is only used to establish the convergence rate whereas establishing the convergence of the discrete scheme \eqref{disc-obs} can be obtained under the standard hypothesis on the continuous solution as in the following theorem.

\begin{theorem}[Convergence]\label{cor-conv}
Assume that $\O\subset \RR^d\; (d>1)$ is a bounded domain with the boundary $\dr\O$, $f \in L^2(\O)$, $\psi \in C^2(\O)\cap C(\overline \O)$, $\psi \geq 0$ on $\dr\O$, and let $(\disc_m)_{m\in\NN}$ be a sequence of a gradient discretisation, that is coercive, consistent, limit--conforming and compact. Let $\bar c$ be the solution to \eqref{weak-obs}. If $\cK_{\disc_m}$ is a non empty set, then there exists a unique solution $u_m \in \cK_{\disc_m}$ to the discrete problem \eqref{disc-obs} (with $\disc=\disc_m$), and, as $m \to \infty$,
\begin{itemize}
\item $\Pi_{\disc_m}c_m$ converges strongly to $\bar c$ in $L^2(\O)$,
\item $\nabla_{\disc_m}c_m$ converges strongly to $\nabla\bar c$ in $L^2(\O)^d$, and
\item $\Delta_{\disc_m}c_m$ converges strongly to $\Delta\bar c$ in $L^2(\O)$.
\end{itemize}
\end{theorem}

\begin{proof}
Let $c:=c_m$ and $\varphi:=I_{\disc_m}\bar c \in \cK_{\disc_m}$ in \eqref{disc-obs}, where $I_\disc$ is defined by \eqref{interp-eq} with $\disc=\disc_m$. Use the discrete relation \eqref{eq-ponc} to obtain
\[
\begin{aligned}
\|\Delta_{\disc_m}c_m\|_{L^2(\O)}^{2}
&\leq \| f \|_{L^2(\O)} \| \Pi_{\disc_m}(c_m - I_{\disc_m}\bar c) \|_{L^2(\O)} \\
&\quad+\|\Delta_{\disc_m}c_m\|_{L^2(\O)} \|\Delta_{\disc_m}I_{\disc_m}\bar c\|_{L^2(\O)}\\
&\leq \cC_P \| f \|_{L^2(\O)} \| \Delta_{\disc_m}(c_m - I_{\disc_m}\bar c) \|_{L^2(\O)} \\
&\quad+\|\Delta_{\disc_m}c_m\|_{L^2(\O)} \|\Delta_{\disc_m}I_{\disc_m}\bar c\|_{L^2(\O)},
\end{aligned}
\]
where $\cC_P$ does not depend on $m$. Thanks to Young’s inequality, the above inequality gives
\begin{equation}\label{eq-bound-1}
\|\Delta_{\disc_m}u_m\|_{L^2(\O)}^{2}
\leq 
C\left( \| f \|_{L^2(\O)} + \| \Delta_{\disc_m}I_{\disc_m}\bar c) \|_{L^2(\O)} \right),
\end{equation}
where $C$ is also independent on $m$. By the triangle inequality and \eqref{interpl}, one writes
\[
\begin{aligned}
\| \Delta_{\disc_m}I_{\disc_m}\bar c \|_{L^2(\O)}&\leq 
\| \Delta_{\disc_m}I_{\disc_m}\bar c -\Delta\bar c \|_{L^2(\O)} + \| \Delta\bar c \|_{L^2(\O)}\\
&\leq \cS_{\disc_m}(\bar c) + \| \Delta\bar c \|_{L^2(\O)}. 
\end{aligned}
\]
The consistency of $\disc_m$ and the standard regularity $\Delta\bar c \in L^2(\O)$ show that the quantity $\| \Delta_{\disc_m}I_{\disc_m}\bar c \|_{L^2(\O)}$ is bounded and thus Estimate \eqref{eq-bound-1} proves that $\|\Delta_{\disc_m}c_m\|_{L^2(\O)}$ remains bounded.

The regularity results of the limit for the second order gradient discretisation stated in \cite[Lemma 2.15]{S1} can easily be extended to the fourth order gradient discretisation. These results and the boundedness of $\|\Delta_{\disc_m}c_m\|_{L^2(\O)}$ established above show that there exists $\bar c\in H_0^2(\O)$, such that $\Pi_{\disc_m}\bar c_m \rightharpoonup \bar c$ in $L^2(\O)$, $\nabla_{\disc_m}\bar c_m \rightharpoonup \nabla\bar c$ in $L^2(\O)^d$ and $\Delta_{\disc_m}\bar c_m \rightharpoonup \Delta\bar c$ in $L^2(\O)$. From the compactness property of $\disc_m$, we see that $\Pi_{\disc_m}\bar c_m$ converges strongly to $\bar c$ in $L^2(\O)$. Now, since $c_m\in\cK_\disc$, we obtain $\Pi_{\disc_m}c_m \leq \psi$, which yields $\bar c \in \cK$.

Let us now show that $\bar c$ satisfies \eqref{weak-obs}. For any $v \in \cK$, the interpolant defined in \eqref{interp-eq} and the consistency property of $(\disc_m)_{m\in\NN}$ imply $\Pi_{\disc_m}I_{\disc_m}v \rightarrow v$ in $L^2(\O)$, $\nabla_{\disc_m}I_{\disc_m}v \rightarrow \nabla v$ in $L^2(\O)^d$, and $\Delta_{\disc_m}I_{\disc_m}v \rightarrow \Delta v$ in $L^2(\O)$. Choosing $c:=c_m$ and $\varphi:=I_{\disc_m}v \in \cK_{\disc_m}$ in \eqref{disc-obs} and pass to the limit show that \eqref{weak-obs} is satisfied for any $\varphi \in \cK$, which concludes that $c$ is the continuous solution, thanks to the strong--weak convergences of sequences. 

In order to obtain the strong convergence of $\Delta_{\disc_m}c_m$ and $\nabla_{\disc_m}c_m$, we take $\varphi:=I_{\disc_m}\bar c$ as a generic function in \eqref{disc-obs} for $\disc_m=\disc$. It implies, due to $c$ is the solution to \eqref{weak-obs} 
\[
\begin{aligned}
0 &\leq \|\Delta_{\disc_m}c_m-\Delta\bar c\|_{L^2(\O)}^{2}\\
&=\dsp\int_\O (\Delta_{\disc_m}c_m(\x)-\Delta\bar c(\x))^2 \ud \x\\
&\leq \dsp\int_\O f(\x)\Pi_{\disc_m}(c_m -I_{\disc_m}\bar c)(\x) \ud \x
+\dsp\int_\O \Delta\bar c(\x)^2 \ud \x.
\end{aligned}
\]
Pass to the limit in this inequality and use again the strong--weak convergence of sequences to conclude the strong convergence of $\Delta_{\disc_m}c_m$. Finally, introduce $\Delta_{\disc_m}I_{\disc_m}\bar c$ and use the triangular inequality to attain, thanks to \eqref{corc-obs} and \eqref{interpl}
\[
\begin{aligned}
\|\nabla_{\disc_m}c_m&- \nabla\bar c\|_{L^2(\O)^d}\\
&\leq \|\nabla_{\disc_m}(c_m-I_{\disc_m}\bar c)\|_{L^2(\O)^d}
+\|\nabla_{\disc_m}I_{\disc_m}\bar c - \nabla\bar c\|_{L^2(\O)^d}\\
&\leq \cC_{\disc_m} \|\Delta_{\disc_m}(c_m- I_{\disc_m}\bar c)\|_{L^2(\O)}
+\cS_{\disc_m}(\bar c).
\end{aligned}
\]
Since $(\disc_m)_{m\in\NN}$ is coercive and consistent, the strong convergence of $\Delta_{\disc_m}c_m$ established above therefore leads to $\cC_{\disc_m} \|\Delta_{\disc_m}(c_m- I_{\disc_m}\bar c)\|_{L^2(\O)}
+\cS_{\disc_m}(\bar c) \to 0$, as $m\to \infty$, which completes the proof.
\end{proof}


\bibliographystyle{siam}
\bibliography{pvi-ref}

\end{document}